\documentclass{amsart}

\usepackage{amssymb}
\usepackage{latexsym}
\usepackage{amsmath}
\usepackage{euscript}
\usepackage{graphics}
\usepackage{tikz}

      \def\dC{{\mathbb C}}

   \def\dN{{\mathbb N}}   
      \def\dR{{\mathbb R}}

   \def\dZ{{\mathbb Z}}

   \def\cW{{\mathcal W}}

\def\bm\chi{\mbox{\boldmath$\chi$}}

\let\xker=\ker \def\ker{{\xker\,}}

\def\supp{{\rm supp\,}}

\newcommand{\HP}{Her\-mite-Pad\'e}
\newcommand{\sddots}{\begin{picture}(2,2)
\multiput(0,0)(1.5,1){3}{.}\end{picture}}

\def\deg{\operatorname{deg}}

\newtheorem{theorem}{Theorem}[section]
\newtheorem{proposition}[theorem]{Proposition}

\theoremstyle{remark}

\newtheorem{remark}[theorem]{Remark}

\numberwithin{equation}{section}

\newenvironment{thmbis}
  {\addtocounter{theorem}{-1}%
   \begin{theorem}}
  {\end{theorem}}

\author{Alexander I. Aptekarev}
\address{
Alexander I. Aptekarev\\
Keldysh Institute for Applied Mathematics\\
Russian Academy of Sciences\\
Miusskaya pl. 4\\
125047 Moscow, RUSSIA}

\author{Maxim Derevyagin}
\address{
Maxim Derevyagin\\
University of Mississippi\\
Department of Mathematics\\
Hume Hall 305 \\
P. O. Box 1848 \\
University, MS 38677-1848, USA }
\email{derevyagin.m@gmail.com}

\author{Walter Van Assche}
\address{
Walter Van Assche\\
KU Leuven\\
Department of Mathematics\\
Celestijnenlaan 200B box 2400\\
BE-3001 Leuven}

\date{\today}


\begin{document}

\title[Discrete integrable systems generated by HP approximants]{Discrete integrable systems generated by
Hermite-Pad\'e approximants}

\begin{abstract}
We consider \HP{} approximants in the framework of
discrete integrable systems defined on the lattice $\dZ^2$. We show that the concept of multiple orthogonality is intimately related to the Lax representations for the entries of the nearest neighbor recurrence relations and it thus gives rise to a discrete integrable system. We show that the converse statement is also true. More precisely, given the discrete integrable system in question there exists a perfect system of two functions, i.e., a system for which the entire table of \HP{} approximants exists. In addition, we give a few algorithms to find solutions of the discrete system. 

\end{abstract}

\maketitle

\section{Introduction}\label{sec:1}

Nowadays modern technologies allow us to handle an enormous amount of information. As a consequence of this development, it is in many instances more advantageous to face the analysis of discrete data rather than continuous data. For this reason we are witnessing that the world in this century
requires more and more the understanding of discrete models and that is why we decided to concentrate our attention on studying discrete models that take their origin in \textit{orthogonality}, one of the basic mathematical concepts.
Mathematically speaking, the discrete models we are going to consider are \textit{systems of difference equations}. Recent advances in a number of mathematical fields reveal that discrete systems are in many respects even more fundamental than continuous ones (for instance see \cite{Nov},  \cite{S2010}).

In this paper we follow the streamline of \textit{discrete integrable systems} (see \cite{B2004}, \cite{BS2002}) and our main interest is in discrete systems on $\dZ^2$ represented by a field of square  invertible  matrices
\begin{equation}\label{0 DIS}
L_{n,m}\,, \,\, M_{n,m}\,\in \,\dC^{d\times d}\, , \qquad n,m\in\dZ,
\end{equation}
which satisfy the \textit{discrete zero curvature} condition (or form \textit{a Lax pair}) on $\dZ^2$:
\begin{equation}\label{0 ZCC}
L_{n,m+1}M_{n,m}-M_{n+1,m}L_{n,m}=0\,.
\end{equation}
The elements of the discrete system \eqref{0 DIS} are transition matrices which define the evolution of the \textit{wave function} $\Psi_{n,m}$
\begin{equation}\label{0 WF}
\Psi_{n+1,m}(z)=L_{n,m}(z)\Psi_{n,m}(z), \quad \Psi_{n,m+1}(z)=M_{n,m}(z)\Psi_{n,m}(z),
\end{equation}
and the condition \eqref{0 ZCC} describes the consistency or integrability of the equations  \eqref{0 WF}. In turn, the relations \eqref{0 ZCC} represent a nonlinear system of difference equations.

Our findings are mainly inspired by the connection between discrete integrable systems, \textit{orthogonal polynomials}, 
and \textit{Pad\'e approximants}
(see \cite{PGR1995}, \cite{SNderK2011}). For example, the discrete dynamics
\begin{equation}\label{0 DTT}
x^k\, d\mu(x), \qquad k\in\dZ_+
\end{equation}
of the measure $d\mu$ supported on $[0,+\infty)$ generates a family of orthogonal polynomials $\{P_n^{(k)}(x)\}, \,\, \mbox{deg }P_n^{(k)}=n$. These polynomials also appear as numerators of Pad\'e approximants in the Pad\'e table and some nearest neighbour polynomials $P_n^{(k)}$ are related
$$
P_{n+1}^{(k)}(x)  =  xP_{n}^{(k+1)}(x)-V_{n}^{(k)}P_{n}^{(k)},  \qquad
P_{n+1}^{(k)}(x)  =  xP_{n}^{(k+2)}(x)-W_{n}^{(k)}P_{n}^{(k+1)},
$$
where the coefficients of the relations are expressed by means of the Hankel determinants
\begin{equation*}
V_{n}^{(k)}  =  \frac{S_{n+1}^{(k+1)}S_{n}^{(k)}}{S_{n}^{(k+1)}S_{n+1}^{(k)}} ,\quad
W_{n}^{(k)}  =  \frac{S_{n+1}^{(k+1)}S_{n}^{(k+1)}}{S_{n+1}^{(k)}S_{n}^{(k+2)}} ,\quad
S_{n+1}^{(k)}=
\left|\begin{array}{cccc}
s_{k}           & \ldots      &  s_{n+k}  \\
\vdots          & \sddots            &  \vdots  \\
s_{n+k}      & \ldots     &  s_{2n+k}
\end{array}\right|,
\end{equation*}
and $s_j=\int_{0}^{\infty}x^j\, d\mu(x)$.
Finally, the consistency of these relations gives the discrete zero curvature condition \eqref{0 ZCC} for the discrete system \eqref{0 DIS} of $2 \times 2$ matrices
\begin{equation*}
L_{n,k}=\left( \begin{array}{cc}
-V_{n}^{(k)} & x \\
-V_{n}^{(k)} & x + W_{n}^{(k)} -V_{n}^{(k+1)}
\end{array} \right), \quad
M_{n,k}=\frac{1}{x}\left( \begin{array}{cc}
0 & x \\
- V_{n}^{(k)} & x+ W_{n}^{(k)}
\end{array} \right).
\end{equation*}
Recall that in the theory of Pad\'e approximants this discrete system becomes
 the quotient-difference algorithm, and in integrable systems theory it leads to the discrete-time Toda equation
(see, e.g., \cite{Suris}).

\medskip

In the present paper we introduce a new
 class of discrete integrable systems of $3 \times 3$ matrices  \eqref{0 DIS}--\eqref{0 ZCC}. The construction  of these systems is based on the  theory of \textit{Hermite-Pad\'e rational approximants}, which were introduced by Hermite \cite{Her} in
connection to his outstanding proof of the transcendence of $e$. These days this theory is known to play an important role in various fields ranging from number theory \cite{Aper}, \cite{Apt2011}, \cite{vanA2001} to random matrix theory \cite{K2010}, \cite{AptKu}.

To proceed, let us briefly consider the concept of Hermite-Pad\'e rational approximants (for details, see the surveys \cite{Apt}, \cite{vanA1999}). Let $\vec{f} = (f_1,f_2)$ be a vector of Laurent series at infinity
\begin{equation}  \label{0 f}
   f_j(z) = \sum_{k=0}^\infty \frac{s_{j,k}}{z^{k+1}},\qquad j=1,2.
\end{equation}
The \textit{\HP{} rational approximants} (of type II)
\[  \pi_{\vec{n}} = \left( \frac{Q_{\vec{n}}^{(1)}}{P_{\vec{n}}},
    \frac{Q_{\vec{n}}^{(2)}}{P_{\vec{n}}} \right) \]
for the vector $\vec{f}$ and multi-index $\vec{n} = (n_1,  n_2) \in \mathbb{N}^2$
are defined by
\[  \deg P_{\vec{n}} \leq |\vec{n}| = n_1 + n_2, \]
\begin{equation}   \label{0 HP}
    f_j(z)P_{\vec{n}}(z) - Q_{\vec{n}}^{(j)}(z) =:
    R_{\vec{n}}^{(j)}(z) = \mathcal{O}\left(\frac{1}{z^{n_j+1}}\right), \qquad z \to \infty,
\end{equation}
where the $Q_{\vec{n}}^{(j)}$ are polynomials,
for $j=1, 2$.
This definition is equivalent to a homogeneous linear system of equations
for the coefficients of the polynomial $P_{n_1,n_2}$.
This system always has a solution, but the solution is not necessarily unique.
In the case of uniqueness (up to a multiplicative constant) and in
case any non-trivial solution has full degree $ \mbox{deg}\,P_{n_1,n_2}= n_1+n_2$,
the multi-index $(n_1,n_2)$ is called \textit{normal} and the polynomial $P_{\vec{n}}$ can
be normalized to be monic.

Clearly these polynomials can be put in a table $\{P_{n,m}\}$. If all indices of this table are normal, then the system of functions \eqref{0 f} is called a \textit{perfect system}. The notion of perfect systems was introduced by Mahler \cite{Mah}. For perfect systems, the \HP{} polynomials \eqref{0 HP} satisfy a system of recurrence
relations
\begin{equation} \label{0 RR}
\begin{cases}
    \,\,\,P_{n+1,m}(x) = (x-c_{n,m})P_{n,m}(x) - a_{n,m} P_{n-1,m}(x) - b_{n,m} P_{n,m-1}(x),  \\
    \,\,\,P_{n,m+1}(x) = (x-d_{n,m})P_{n,m}(x) - a_{n,m} P_{n-1,m}(x) - b_{n,m} P_{n,m-1}(x),
\end{cases}
\end{equation}
with $a_{0,m}=b_{n,0}=0$ for all $n,m \geq 0$ and $a_{n,m}\neq 0$, $b_{n,m}\neq 0$ for all other indices $(n,\,m)$.
As we will see, the consistency of these relations also leads to Lax pair representations, where the corresponding matrices $L_{n,m}$ and $M_{n,m}$ have the forms
\begin{equation} \label{Int_eq:2.8}
    L_{n,m}  = \begin{pmatrix}
                  x+ \alpha_{n,m}^{(1)}& \alpha_{n,m}^{(2)} & \alpha_{n,m}^{(3)} \\
                  \alpha_{n+1,m}^{(4)} & 0 &  0 \\
                  \alpha_{n+1,m}^{(5)} & 0 &  1
                 \end{pmatrix},\qquad
    M_{n,m}  = \begin{pmatrix}
                  x+\beta_{n,m}^{(1)} & \alpha_{n,m}^{(2)} & \alpha_{n,m}^{(3)} \\
                  \alpha_{n,m+1}^{(4)} & 1 &  0 \\
                  \alpha_{n,m+1}^{(5)} & 0 &  0
                 \end{pmatrix},
\end{equation}
where the entries are related to the coefficients of the recurrence relations \eqref{0 RR} for the \HP{} polynomials of the functions $f_1$ and $f_2$ as follows:
\begin{equation*}
    c_{n,m} = -\alpha_{n,m}^{(1)}, \quad
    d_{n,m} = -\beta_{n,m}^{(1)}, \quad
\   a_{n,m} = -\alpha_{n,m}^{(4)}\alpha_{n,m}^{(2)}, \quad b_{n,m} = -\alpha_{n,m}^{(5)}\alpha_{n,m}^{(3)}.
\end{equation*}
Both sets of coefficients  of the relations \eqref{0 RR} and of entries of the matrices \eqref{Int_eq:2.8} can be represented by  the power series coefficients  of the perfect system of functions \eqref{0 f}
\begin{equation} \label{0 a-s-alpha}
(a,\,b,\,c,\,d)_{n,m}\quad \longleftarrow \quad
\{s_{j,k}\}_{j=1,2} \quad \longrightarrow \quad 
\bigl( \alpha^{(1)}, \beta^{(1)}, \alpha^{(2)}, \ldots,\alpha^{(5)} \bigr)_{n,m}
\end{equation}
(details will be given below).
The main message of this note is to show that there are discrete integrable systems related to \HP{} approximation and the theory of such approximants allows  to construct solutions of these systems provided the initial boundary data are given. In this paper we restrict ourselves to the simplest case of \HP{} approximants generated by two functions, for which we have the following result. 
\begin{theorem} \label{T0 1}
The zero curvature condition \eqref{0 ZCC}
holds for a family of $3\times 3$ transition matrices $L_{n,m}$ and $M_{n,m}$ of the form \eqref{Int_eq:2.8} if and only if there is a perfect system of two functions \eqref{0 f} such that  $P_{n,m}$ are the \HP{} polynomials with the coefficients of the recurrence relations \eqref{0 RR} corresponding to \eqref{0 a-s-alpha}.
\end{theorem}
The proof of this statement is immediate from Proposition \ref{Pr2}, which is a slight generalization of the result from \cite{vanA2011}, and Proposition \ref{Prop4.5}, which is the converse statement and is the basis for further development of the present paper.

Since the transition matrices are explicitly known, it is not so difficult to re-rewrite the discrete zero curvature condition in terms of the coefficients of \eqref{0 RR}. Thus, we have arrived at the following statement, which is simply an equivalent form of Theorem \ref{T0 1}. 
\begin{thmbis} \label{T0 2}
 The discrete Lax pair equations \eqref{0 ZCC} for the matrices $L_{n,m}$, $M_{n,m}$ of the form \eqref{Int_eq:2.8} are equivalent to the nonlinear system of difference equations for the coefficients of the recurrence relations \eqref{0 RR}
\begin{equation}  \label{0 DiffEq}
\begin{cases}
    \,\,\,c_{n,m+1}\,=\, c_{n,m}\,+\,\displaystyle\frac{(a+b)_{n+1,m}\,-\,(a+b)_{n,m+1}}{(c-d)_{n,m}}  \\
    \,\,\,d_{n+1,m}\,=\, d_{n,m}\,+\,\displaystyle\frac{(a+b)_{n+1,m}\,-\,(a+b)_{n,m+1}}{(c-d)_{n,m}}  \\
    \,\,\,a_{n,m+1}\,=\,a_{n,m} \,\displaystyle\frac{(c-d)_{n,m}}{(c-d)_{n-1,m}} \\
    \,\,\,b_{n+1,m}\,=\,b_{n,m} \,\displaystyle\frac{(c-d)_{n,m}}{(c-d)_{n,m-1}}
\end{cases} \qquad n,m \geq 0\,,
\end{equation}
where the initial data $(c,\,a)_{n,0}$, $(d,\,b)_{0,m}$ are supposed to be given and the coefficients also satisfy the boundary conditions $a_{0,m} = 0 = b_{n,0}$. Moreover, the system has a solution such that  $a_{n,m}\neq 0$, $b_{n,m}\neq 0$ for $n,m>0$ if and only if the initial data correspond to a perfect system.
\end{thmbis}

This theorem explicitly gives us the underlying boundary value problem. Moreover, the way it is written, one can easily get an idea about the continuum limit of this discrete system.

Once the boundary value problem is given, it is natural to look for its solution. To this end, we present here a new method of solving the discrete system using branched continued fractions related to \HP{} approximants. Furthermore, these branched continued fractions are introduced here for the first time and they appear as the multiple orthogonal adaptation of the classical continued fraction representation \eqref{toGen} that solves some inverse problems for finite Jacobi matrices \cite{GS}. Also, it is used in dynamical systems and asymptotic analysis.

Another natural question is whether the initial data lead to a well-posed problem. In order to understand this issue a bit more deeper, let us notice that \HP{} approximants are intimately related to the notion of multiple orthogonal polynomials.
If the coefficients of the Laurent series (\ref{0 f}) are the moments of positive measures $\mu_1$ and  $\mu_2$
supported on $\mathbb{R}$
\begin{equation}   \label{0 fMOP}
   f_j(z) = \sum_{k=0}^\infty \frac{s_{j,k}}{z^{k+1}} = \int_{\mathbb{R}} \frac{d\mu_j(x)}{z-x}, \qquad
  s_{j,k} = \int_{\mathbb{R}} x^k\, d\mu_j(x),
\end{equation}
then the \HP{} denominators $P_{n_1,n_2}$ from \eqref{0 HP} satisfy
\begin{equation}  \label{0 MOP}
\int P_{n_1,n_2}(x) x^k\, d\mu_j(x) = 0, \qquad k=0,1,\ldots,n_j-1,\qquad  j=1,2.
\end{equation}
Polynomials  defined by the system of orthogonality relations \eqref{0 MOP} are called \textit{multiple orthogonal polynomials.} The idea of this concept is the following: given two measures $(\mu_1, \mu_2)$, we distribute the $n_1+n_2$ orthogonality relations between
these measures and aim to find a monic polynomial $P_{n_1,n_2}$ of degree $ \mbox{deg}\,P_{n_1,n_2}=n_1+n_2$ that is orthogonal to the $n_1$ first monomials with respect to the first measure and to the $n_2$ first monomials with respect to the second measure.

The multiple orthogonal polynomials (i.e., the  \HP{} polynomials) inherit all the remarkable properties for Hermite-Pad\'e approximants, like existence  of the monic polynomials of full degree for the normal indices and the recurrence relations \eqref{0 RR}, which were obtained for the first time in 
\cite{vanA2011} for the multiple orthogonal polynomials. In the context of our paper we use these polynomials to generate a general class of perfect systems for which the corresponding tables of multiple orthogonal polynomials exist entirely.

One should not think that any two measures form a perfect system. It is actually not a trivial task to check whether one can obtain such a table for any two measures. Moreover, there are measures for which it is impossible to define polynomials for all indices. So, in order to address this issue we give a short review of Angelesco and Nikishin systems at the end of Section 5, which are actually pairs of measures for which the corresponding table of multiple polynomials exist entirely and, therefore, the boundary data obtained from such systems lead to well-posed problems for the discrete system in question. Some algorithms for solving the boundary value problem will be mentioned in Section 5 as well.

To conclude the introduction we want to remark that it is also possible to consider an analogue of the dynamics \eqref{0 DTT} for multiple orthogonal polynomials and get a discrete integrable system similar to the qd-algorithm. Moreover, the qd-algorithm was already partially adapted to the case of multiple orthogonal polynomial in \cite{I2003}. This will be considered in more detail elsewhere. 
\bigskip

\noindent{\bf Structure of the paper.} The following two sections serve as an introduction to the topic. In particular we give in Section~\ref{sec:2}  more explanations about general discrete integrable systems. Then, in Section~\ref{sec:3}, we consider some known $2 \times 2$ matrix relations from the theory of orthogonal polynomials and continued fractions, i.e.,
the theory that concerns classical diagonal Pad\'e approximants. A generalization of these relations to the $3 \times 3$ matrix case for \HP{} approximants and multiple orthogonal polynomials, which plays a decisive role for establishing the connection to discrete integrable systems represented by $3 \times 3$ matrices, is presented in Section~\ref{sec:4}. Particularly, in that section we give and prove several propositions, which lead to a proof of Theorem~\ref{T0 1} and, in turn, Theorem \ref{T0 2}. Finally, in Section 5 we provide the reader with a generic class of perfect systems such as Angelesco and Nikishin systems. These systems generate the boundary data for which the discrete integrable system is solvable.

\bigskip

\noindent{\bf Acknowledgements.} A.I. Aptekarev was supported by grant RScF-14-21-00025. M. Derevyagin thanks the hospitality of the Department of Mathematics of KU Leuven, where his part of the research was mainly done while he was a postdoc there. M. Derevyagin and W. Van Assche gratefully acknowledge the support of FWO Flanders project G.0934.13, KU Leuven research grant OT/12/073 and the Belgian Interuniversity Attraction Poles 
programme P07/18.

\section{The generic Lax representations}\label{sec:2}

Here we recall some basic notions in the theory of discrete integrable systems following \cite{BSbook}
(see also \cite{Adler2001}, \cite{BS2002}, \cite{PGR1995}, and \cite{SNderK2011}).

Let us consider a regular square lattice $\dZ^2$, that is the set of all pairs $(n,m)$ of integer numbers $n$ and $m$.
The main object of our study is {\it wave functions} $\Psi_{n,m}$ defined on all the vertices $(n,m)$ of $\dZ^2$
and having their values in $\dC^{k\times k}$ (for simplicity, we restrict ourselves here to the cases $k=2$ and $k=3$).
The wave function $\Psi_{n,m}$ depends on a complex parameter $z$, which is interpreted as the spectral
parameter. We assume that for any oriented edge the values of the wave function at the vertices that this edge connects are related via
{\it transition matrices} $L_{n,m}$ and $M_{n,m}$ as follows
\[
\Psi_{n+1,m}(z)=L_{n,m}(z)\Psi_{n,m}(z), \quad \Psi_{n,m+1}(z)=M_{n,m}(z)\Psi_{n,m}(z).
\]
We always require that the transition matrices are  invertible and therefore one has
 \[
\Psi_{n,m}(z)=L_{n,m}^{-1}(z)\Psi_{n+1,m}(z), \quad \Psi_{n,m}(z)=M_{n,m}^{-1}(z)\Psi_{n,m+1}(z).
\]
It is clear that the value of the wave function must not depend on the path one takes to get to the corresponding vertex. Thus, in order that the wave function $\Psi_{n,m}$ is well defined, the following {\it zero curvature condition}
must be satisfied
\begin{equation}\label{Lax}
L_{n,m+1}M_{n,m}-M_{n+1,m}L_{n,m}=0, \qquad n,m\in\dZ.
\end{equation}
As is known, the zero curvature condition is equivalent to  integrability. Thus,
a discrete system that admits the representation \eqref{Lax} is called integrable \cite{Adler2001}, \cite{B2004}, \cite{BS2002}.

Before going further, let us take a careful look at the zero curvature condition. Observe at first that
one can rewrite \eqref{Lax} as
\[
L_{n,m}^{-1}M_{n+1,m}^{-1}L_{n,m+1}M_{n,m}=I.
\]
Next, from the following picture
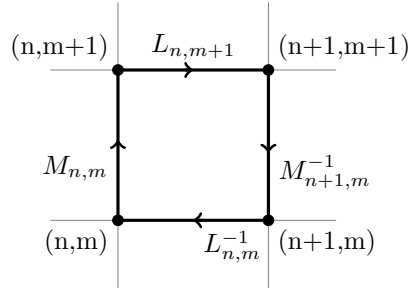
\begin{center}
\begin{tikzpicture}
\draw[step=2cm,gray,very thin] (-0.9,-0.9) grid (2.9,2.9);
\draw[very thick] (0,0) rectangle (2,2);
\filldraw[black] (0,0) circle (2pt) node[anchor=north east] {(n,m)};
\draw[very thick, ->] (0,1)node[anchor=north east] {$M_{n,m}$};
\filldraw[black] (0,2) circle (2pt) node[anchor=south east] {(n,m+1)};
\draw[very thick, ->] (0.9,2)--(1,2) node[anchor=south] {$L_{n,m+1}$};
\filldraw[black] (2,2) circle (2pt) node[anchor=south west] {(n+1,m+1)};
\draw[very thick, <-] (2,0.9)--(2,1) node[anchor=north west] {$M_{n+1,m}^{-1}$};
\filldraw[black] (2,0) circle (2pt) node[anchor=north west] {(n+1,m)};
\draw[very thick, ->] (1.1,0)--(1,0) node[anchor=north west] {$L_{n,m}^{-1}$};
\end{tikzpicture}
\end{center}
we see that the zero curvature condition implies that the product of the transition matrices along
the oriented simple square path on $\dZ^2$ beginning at the vertex $(n,m)$ is the identity matrix.
This observation can be immediately extended to the case of domino paths by reducing them to the just considered simplest case.
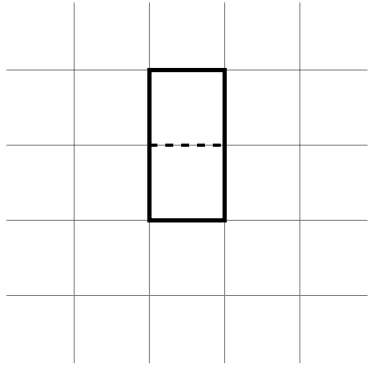
\begin{center}
\begin{tikzpicture}
\draw[step=1cm,gray,very thin] (-1.9,-1.9) grid (2.9,2.9);
\draw[ultra thick] (0,0) rectangle (1,2);
\draw[dashed, very thick] (0,1)--(1,1);
\end{tikzpicture}
\end{center}
Now it is clear how to generalize this statement to the case of any closed oriented path on $\dZ^2$.
Thus the condition \eqref{Lax} means that if one fixes a closed oriented path on the lattice $\dZ^2$, then
the product of the transition matrices in the order they appear along the path must be the identity matrix.
Note that this property resembles the Cauchy theorem for holomorphic functions and, therefore, it can be considered
as its noncommutative multiplicative analogue for functions on $\dZ^2$.
The relation \eqref{Lax} is also called the Lax representation and this is one of the ways to say that the underlying
discrete system is integrable.

It turns out that different types of wave functions appear in the theory of orthogonal polynomials and they  are very useful to achieve a big variety of goals. However, it has to be pointed out that the discrete integrable systems and wave functions, that have their origin in orthogonality, are naturally defined on $\dN^2$, where $\dN=\{1,2,3, \dots\}$.
Luckily one can appropriately extend them to $\dZ_+^2$ or even to $\dZ^2$ depending on the needs.

\section{Orthogonal polynomials via $2\times 2$ matrix polynomials}\label{sec:3}

In this section we briefly go over the ideas in the theory of ordinary orthogonal polynomials in order to consider more general objects in Section~\ref{sec:4} and to get a discrete integrable system whose Lax pair is expressed via $3\times 3$ matrices that come out in a natural way in the situation that is of interest in this paper..

\subsection{The Schur-Euclid algorithm}\label{sec:31}

Suppose we are given a nontrivial Borel measure $d\mu$ on the real line $\dR$. Assume also that all the moments of the measure $d\mu$
are finite. Then the Schur algorithm, which is a straightforward generalization of Euclid's algorithm, leads
to the following continued fraction
\begin{equation*}
\varphi(z)=\int_{\dR}\frac{d\mu(t)}{t-z}\sim
-\frac{1}{\displaystyle{z-a_0-\frac{b_0^2}{\displaystyle{z-a_1-\frac{b_1^2}{\ddots}}}}},
\end{equation*}
where $b_j^2>0$ and $a_j\in\dR$ for $j=0,1,\dots$. This continued fraction is called a $J$-fraction and such a representation actually exist for a larger class of analytic functions.

It is natural to consider continued fractions as infinite sequences of linear fractional transformations. In particular,
in the case of the $J$-fraction, one has the following sequence
\begin{equation*}
    \varphi_j(z)=T_j\left(\varphi_{j+1}(z)\right)=-\frac{1}{z-a_j+
    b_j^2\varphi_{j+1}(z)}, \qquad j\in\dZ_+,
\end{equation*}
with the initial condition $\varphi_0=\varphi$.
Also, it is well known that a linear fractional transformation can be represented as a $2\times 2$ matrix, i.e.
\begin{equation*}
T_j\mapsto\cW_j(z)=\begin{pmatrix}0 & -\frac{1}{b_j}\\
                            b_j &  \frac{z-a_j}{b_j}
                            \end{pmatrix},\qquad j\in\dZ_+.
\end{equation*}
Let us now introduce matrices corresponding to the approximants for the $J$-fraction, that is, the finite truncations
of the continued fraction:
\begin{equation}\label{helpW}
\cW_{[n,0]}(z)=\cW_0(z)\cW_1(z)\dots\cW_n(z), \qquad n\in\dZ_+.
\end{equation}
Before showing how to construct a set of a wave functions and transition matrices on $\dZ^2$ let us see what the elements
of the matrix polynomial $\cW_{[n,0]}$ are. To this end, we put
\begin{equation*}
   \left(%
\begin{array}{c}
  -Q_0 \\
  P_0 \\
\end{array}%
\right):=\left(%
\begin{array}{c}
  0 \\
  1 \\
\end{array}%
\right),\quad
\left(%
\begin{array}{c}
 -Q_{j+1}(z) \\
 P_{j+1}(z) \\
\end{array}%
\right):=\cW_{[j,0]}(z)\left(%
\begin{array}{c}
  0 \\
  1 \\
\end{array}%
\right), \quad j\in\dZ_+ .
\end{equation*}
Then taking into account the relation $\cW_{[j,0]}(z)=\cW_{[j-1,0]}(z)\cW_{j}(z)$ we also have that
\begin{equation*}
\cW_{[j,0]}(z)\left(%
\begin{array}{c}
  1 \\
  0 \\
\end{array}%
\right)=
\cW_{[j-1,0]}(z)\left(%
\begin{array}{c}
  0 \\
  b_j \\
\end{array}%
\right)=
\left(%
\begin{array}{c}
  -b_j Q_j(z) \\
  b_j P_j(z) \\
\end{array}%
\right), \quad j\in\dN.
\end{equation*}
So, the matrix $\cW_{[j,0]}$ has the following form
\begin{equation*}
\cW_{[j,0]}(z)=\left(%
\begin{array}{cc}
  -b_j Q_j(z) & -Q_{j+1}(z) \\
  b_j P_j(z) & P_{j+1}(z) \\
\end{array}%
\right), \quad j\in\dZ_+ .
\end{equation*}
Furthermore, rewriting the relation entrywise
\[
\left(%
\begin{array}{c}
 -Q_{j+1}(z) \\
 P_{j+1}(z) \\
\end{array}%
\right)=\cW_{[j-1,0]}(z)\cW_{j}(z)\left(%
\begin{array}{c}
  0 \\
  1 \\
\end{array}%
\right)=
\frac{1}{b_j}\cW_{[j-1,0]}(z)\left(%
\begin{array}{c}
  -1 \\
  z-a_j \\
\end{array}%
\right), \quad j\in\dN,
\]
we see that the polynomials $P_{j}$, $Q_{j}$ are solutions of the following three-term recurrence relation
\begin{equation}\label{3termOP}
b_{j-1}u_{j-1}(z)+a_ju_{j}(z)+b_{j}u_{j+1}(z)=zu_j(z),\qquad j\in\dN,
\end{equation}
with the initial conditions
\begin{equation*}\label{DiffEq}
 \begin{split}
    P_0(z)&=1,\quad P_1(z)=\frac{z-a_0}{b_0},\\
     Q_0(z)&=0,\quad Q_1(z)=\frac{1}{b_0}.
\end{split}
\end{equation*}
Thus the entries of the matrix $\cW_{[n,0]}$ are orthogonal polynomials of the first and second kind and the corresponding
orthogonality measure is $\mu$. It is worth mentioning that such $2\times 2$ matrix polynomials are extensively used in the theory of moment problems \cite{Ach1961} and that this theory is a particular case of the theory of canonical systems \cite[Chapter 8]{Sakh1997} (see also \cite{Sakh1999}).

To conclude this section, we recall the representation \cite{Khinchin} (see also \cite{GS})
\begin{equation}\label{toGen}
-\frac{b_jP_{j+1}(z)}{P_j(z)}=-\frac{1}{\displaystyle{z-a_j-\frac{b_{j-1}^2}{\displaystyle{z-a_{j-1}-\frac{b_{j-2}^2}{\ddots -\displaystyle{\frac{b_0^2}{z-a_0}}}}}}}
\end{equation}
that will be generalized to the case of multiple orthogonal polynomials and then it will be used in the scheme of finding solutions of the discrete system under consideration.

\subsection{Riemann-Hilbert problems}\label{sec:32}

Here we consider a different interpretation of the Schur-Euclid algorithm in the context of Riemann-Hilbert problems,
which turned out to be quite efficient for asymptotic analysis. Recall that in \cite{FIK1992} a fascinating characterization of
orthogonal polynomial in terms of a Riemann-Hilbert problem was found.
We will explain this characterization here briefly. To this end, we consider a weight function $w$ on $\dR$ that is smooth and has sufficient decay at $\pm\infty$ so that
all the moments $\int_{\dR}x^kw(x)\,dx$ exist. Then the Riemann-Hilbert problem (RHP) consists of the following:
find a $2\times 2$ matrix valued function $Y_n(z)=Y(z)$ such that
\begin{enumerate}
    \item[(i)]
        $Y(z)$ is  analytic for $z\in\dC \setminus \dR$.
    \item[(ii)]
        $Y$ possesses continuous boundary values for $x\in\dR$
        denoted by $Y_{+}(x)$ and $Y_{-}(x)$, where $Y_{+}(x)$ and $Y_{-}(x)$
        are the limiting values of $Y(z')$ as $z'$ approaches $x$ from
        above and below, respectively, and
        \begin{equation}\label{RHPYb}
            Y_+(x) = Y_-(x)
            \begin{pmatrix}
                1 & w(x) \\
                0 & 1
            \end{pmatrix},
            \qquad x\in\dR.
        \end{equation}
    \item[(iii)]
        $Y(z)$ has the following asymptotic behavior at infinity:
        \begin{equation} \label{RHPYc}
            Y(z)= \left(I+ \mathcal{O} \left( \frac{1}{z} \right)\right)
            \begin{pmatrix}
                z^{n} & 0 \\
                0 & z^{-n}
            \end{pmatrix}, \qquad z \to\infty.
        \end{equation}
\end{enumerate}
Before giving the solution of this RHP for $Y$, let us recall that the monic orthogonal polynomials
$\pi_n(z)=z^n+\dots$ satisfy the following three term recurrence relation:
\[
z\pi_j(z)=\pi_{j+1}(z)+a_j\pi_{j}(z)+b_{j-1}^2\pi_{j-1}(z),\qquad j\in\dZ_+.
\]
According to \cite{FIK1992}, the matrix valued function $Y(z)$ given by
    \begin{equation} \label{RHPYsolution}
        Y(z) =
        \begin{pmatrix}
            \pi_n(z) & \frac{1}{2\pi i} \int_{\dR}  \frac{\pi_n(x) w(x)}{x-z}\,dx \\[2ex]
            -2\pi i \gamma_{n-1}^2 \pi_{n-1}(z) & -\gamma_{n-1}^2 \int_{\dR} \frac{\pi_{n-1}(x)w(x)}{x-z}\,dx
        \end{pmatrix}
    \end{equation}
is the unique solution of the RHP for $Y$. Here $\gamma_n$ is the leading coefficient of the corresponding orthonormal polynomial.
Observe that $\det Y$ is an analytic function on $\mathbb{C} \setminus \mathbb{R}$
which has no jump on the real axis. Therefore, $\det Y$ is an entire function. Its behaviour near infinity
is $\det Y(z) = 1 + \mathcal{O}(1/z)$. Thus by Liouville's theorem we find that $\det Y = 1$. Consequently, one can consider
the matrix
\[
  W_{n} = Y_{n+1} Y_{n}^{-1}.
\]
Clearly $W_{n}$ is an analytic function on $\dC \setminus \dR$, and since $Y_{n}$
and $Y_{n+1}$ have the same jump matrix on $\mathbb{R}$ we see that $L_{n}$ has no jump on $\mathbb{R}$.
Hence $W_{n}$ is an entire matrix function. We write the asymptotic condition in the following form
\[
Y_{n}(z) = \left( I + \frac{A(n)}{z} + \mathcal{O}(1/z^2) \right)
		\begin{pmatrix} z^{n} & 0  \\ 0 & z^{-n} \end{pmatrix},
\]
where $A(n,m)$ is the $2\times 2$ matrix coefficient of $1/z$ in the $\mathcal{O}(1/z)$ term.
After some calculations and using Liouville's theorem, we find that
\begin{equation}\label{OP_transM}
W_{n} = \begin{pmatrix}
                  z+A_{1,1}(n+1)-A_{1,1}(n)  & -A_{1,2}(n)   \\
                  A_{2,1}(n+1)  & 0
                 \end{pmatrix}, \qquad n\in\dZ_+.
\end{equation}

\begin{remark}\label{RfactorRH}
In fact we haven't fully used the Riemann-Hilbert problem to recover the wave function
and transition matrices in this case. What we actually exploited is the fact that the solution admits the following factorization
\[
Y_n(z)=R_n(z)\begin{pmatrix}
                1 & \int_{\dR}  \frac{w(x)\,dx}{x-z}  \\
                0 & 1
            \end{pmatrix},
\]
where $R_n(z)$ is a matrix polynomial that has the form
\[
R_n=W_{n-1}\dots W_{0}.
\]
Basically, $R_n$ has a structure similar to that of $\cW_{[n,0]}$ (see formula \eqref{helpW}).
Moreover, $R_n$ coincides with $\cW_{[n,0]}$ up to a constant factor and the inversion.
Now, we can clearly see that what we really need here is the Cauchy transform $\int_{\dR}  \frac{w(x)dx}{x-z}$ and its asymptotic 
behavior at infinity
in order to have \eqref{RHPYc}. Therefore, it is clear that one can repeat all the steps for any Borel measure with finite moments of all orders.
In other words, we have arrived at the Schur-Euclid algorithm:
\begin{enumerate}
    \item[(i)]
        We start with the function
        \[
       Y(z)=Y_0(z)=\begin{pmatrix}
                1 & \int_{\dR}  \frac{d\mu(x)}{x-z}  \\
                0 & 1
            \end{pmatrix},
         \]
         where $\mu$ is a probability Borel measure with finite moments of all orders;
    \item[(ii)]
        Having constructed $Y_n$, we look for the transition matrix $L_{n,0}$ of the form \eqref{OP_transM} such that
        the function
        \[
        Y_{n+1}=W_{n}Y_n
        \]
        obeys the asymptotic condition \eqref{RHPYc}.
\end{enumerate}
Let us emphasize that the transition matrix in step (ii) is uniquely determined due to the construction.

In the next section we will see that Riemann-Hilbert problems
admit generalizations in higher dimensions. Thus, they can serve as a tool to develop the multidimensional
Schur-Euclid algorithm.
\end{remark}

\section{\HP{} and Multiple orthogonal polynomials}\label{sec:4}

Here we present a discrete integrable system associated with a family of \HP{} approximants  and multiple orthogonal polynomials.

\subsection{Two-dimensional recurrence relations}\label{sec:41}

We begin by recalling a generalization of orthogonal polynomials to \HP{} polynomials $P_{n,m}$ for two functions $(f_1,f_2)$, which are analytic in a neighbourhood of infinity. It follows from the Cauchy theorem applied to \eqref{0 HP} that \HP{} polynomials satisfy the orthogonality relations
\begin{equation}  \label{4 HPOR}
\oint_{\Gamma}P_{n_1,n_2}(z) \,z^k\,f_j(z)\, dz = 0, \qquad k=0,1,\ldots,n_j-1,\quad  j=1,2,       
\end{equation}
where the contour $\Gamma:=\partial\Omega$ is the boundary of a domain $\Omega\ni\infty$ in which the functions $f_j\in H(\Omega),\,\,j=1,2$ have holomorphic (analytic and single-valued) continuations. We note that the orthogonality relations \eqref{4 HPOR} are non-Hermitian. They actually become Hermitian when the functions $(f_1,f_2)$ are the Cauchy transforms \eqref{0 MOP} of positive measures $\mu_1,\mu_2$ with compact support on the real line. In this case, the coefficients of the Laurent series \eqref{0 f} for $(f_1,f_2)$ can be considered as the moments of $\mu_1,\mu_2$:
\[
s_k^{(j)} = \oint_{\Gamma} z^k\,f_j(z)\, dz \qquad \longrightarrow \qquad s_k^{(j)} = \int x^k\, d\mu_j(x),\quad j=1,2.
\]

Using the determinant of the coefficients $s_k^{(j)}$
\begin{equation} \label{eq:2.12}
   S_{n,m} = \left| \begin{matrix}
                   s_0^{(1)} & s_1^{(1)} & \cdots & s_{n-1}^{(1)} \\
                   s_1^{(1)} & s_2^{(1)} & \cdots & s_{n}^{(1)} \\
	            \vdots & \vdots & \cdots & \vdots \\
                   s_{n+m-1}^{(1)} & s_{n+m}^{(1)} & \cdots & s_{2n+m-2}^{(1)} \end{matrix}		
           \begin{matrix}
                   s_0^{(2)} & s_1^{(2)} & \cdots & s_{m-1}^{(2)} \\
                   s_1^{(2)} & s_2^{(2)} & \cdots & s_{m}^{(2)} \\
	            \vdots & \vdots & \cdots & \vdots \\
                   s_{n+m-1}^{(2)} & s_{n+m}^{(2)} & \cdots & s_{n+2m-2}^{(2)} 		
                  \end{matrix} \right|,
\end{equation}
we can write a formula for the \HP{} polynomials
\small
 \begin{equation} \label{eq:2.12_1}   P_{n,m}(x) = \frac{1}{S_{n,m}}
          \left|\begin{matrix}
                   s_0^{(1)} & s_1^{(1)} & \cdots & s_{n-1}^{(1)} \\
                   s_1^{(1)} & s_2^{(1)} & \cdots & s_{n}^{(1)} \\
	            \vdots & \vdots & \cdots & \vdots \\
                   s_{n+m}^{(1)} & s_{n+m+1}^{(1)} & \cdots & s_{2n+m-1}^{(1)} \end{matrix}		
           \begin{matrix}
                   s_0^{(2)} & s_1^{(2)} & \cdots & s_{m-1}^{(2)} \\
                   s_1^{(2)} & s_2^{(2)} & \cdots & s_{m}^{(2)} \\
	            \vdots & \vdots & \cdots & \vdots \\
                   s_{n+m}^{(2)} & s_{n+m+1}^{(2)} & \cdots & s_{n+2m-1}^{(2)} 		
                  \end{matrix}
          \begin{matrix} 1 \\ x \\ \vdots \\ x^{n+m} \end{matrix} \right|
\end{equation}
\normalsize
provided that $S_{n,m}$ is nonvanishing. The latter case is a criterion of normality of the index $(n,m)$.
In this paper we assume that
all multi-indices are normal and we investigate the nearest-neighbor recurrence relations.

In \cite{WVAGerKui} a matrix Riemann-Hilbert problem formulation for multiple orthogonal was proposed. Here we slightly generalize this approach for the case of \HP{} polynomials. We can formulate the following Riemann-Hilbert problem:
find a $3\times 3$ matrix function $Y$ such that
\begin{enumerate}
\item[(i)] $Y$ is analytic on $\mathbb{C} \setminus \Gamma,$ i.e., $Y\in H(\Omega)$ and $Y\in H(\mathbb{C} \setminus \overline{\Omega})$,
\item[(ii)] the continuous limits $Y_{+}(x): = \displaystyle\lim_{\Omega\ni\,\xi\rightarrow x\in \,\Gamma } Y(\xi)$,
$Y_{-}(x): = \displaystyle\lim_{\{\mathbb{C} \setminus \overline{\Omega}\}\ni\,\xi\rightarrow x\in \,\Gamma } Y(\xi)$ exist and
\begin{equation} \label{eq:2.1}
      Y_+(x) = Y_-(x) \begin{pmatrix}
                        1 & f_1(x) & f_2(x) \\
                        0 &  1  & 0  \\
                        0 & 0 & 1
                        \end{pmatrix}, \qquad x \in \Gamma,
\end{equation}
\item[(iii)] for $z \to \infty$ one has
\begin{equation} \label{eq:2.2}
     Y(z) = \big( I + \mathcal{O}(1/z) \big) \begin{pmatrix} z^{n+m} & 0 & 0 \\ 0 & z^{-n} & 0 \\ 0 & 0 & z^{-m} \end{pmatrix}.
\end{equation}
\end{enumerate}
Following \cite{FIK1992}, \cite{WVAGerKui} it is easy to show that this Riemann-Hilbert problem
has a unique solution in terms of  the \HP{} polynomials when $(n,m)$, $(n-1,m)$ and $(n,m-1)$
are normal indices, i.e.,
\begin{equation}  \label{eq:2.3}
  Y = \begin{pmatrix}
      P_{n,m} & C_1(P_{n,m}) & C_2(P_{n,m}) \\
      c_1(n,m) P_{n-1,m} & c_1 C_1(P_{n-1,m}) & c_1 C_2(P_{n-1,m}) \\
      c_2(n,m) P_{n,m-1} & c_2 C_1(P_{n,m-1}) & c_2 C_2(P_{n,m-1})
      \end{pmatrix}
\end{equation}
where the Cauchy transform is used
\[
C_j(P) = \frac{1}{2\pi i} \oint_{\Gamma} \frac{P(x)f_j(x)}{x-z}\, dx, \qquad j=1,2,
\]
and the constants $c_1$ and $c_2$ are given by 
$$\frac{-2\pi i}{c_1(n,m)} =  \oint_{\Gamma} P_{n-1,m}(x) x^{n-1}f_1(x)\, dx,  \,\,
\frac{-2\pi i}{c_2(n,m)} =  \oint_{\Gamma}P_{n,m-1}(x) x^{m-1}f_2(x)\, dx. $$

One of the natural outcomes of representing the \HP{} polynomials in the form of Riemann-Hilbert problems
is the nearest-neighbour recurrence relations.

\begin{proposition} \label{thm:1}
Suppose all multi-indices $(n,m) \in\dZ_+^2$ are normal.
Then the \HP{}  polynomials satisfy the system of recurrence
relations \eqref{0 RR}.
\end{proposition}
\begin{proof} As we already mentioned in the introduction, the recurrence relations \eqref{0 RR} for multiple orthogonal polynomials were obtained in \cite{vanA2011}. Here, we will follow the same approach.
Actually, the proof presented in \cite{vanA2011} uses the Riemann-Hilbert problem but, as we see later, the main ingredient of that proof
is a factorization similar to the one revealed in Remark \ref{RfactorRH}. Basically,
the proof goes along the same lines as the construction of the wave function from the Riemann-Hilbert
problem in the case of orthogonal polynomials (see Section 3.2).

Let us start by making the standard observation that $\det Y$ is an analytic function in $\mathbb{C} \setminus \mathbb{R}$
with no jump on the contour $\Gamma$. Hence $\det Y$ is an entire function and its behavior near infinity
is $\det Y(z) = 1 + \mathcal{O}(1/z)$. Thus, by Liouville's theorem we find that $\det Y = 1$.   We can therefore consider
the matrix
\[
L_{n,m} = Y_{n+1,m} Y_{n,m}^{-1},
\]
where the subscript $(n,m)$ is used for the solution \eqref{eq:2.3} of the Riemann-Hilbert problem with
the polynomial $P_{n,m}$ in the entry of the first row and the first column of $Y_{n,m}$.
Clearly $L_{n,m}$ is an analytic function on $\mathbb{C} \setminus \mathbb{R}$, and since $Y_{n,m}$
and $Y_{n+1,m}$ have the same jump matrix on $\mathbb{R}$ we see that $L_{n,m}$ has no jump on $\mathbb{R}$.
Hence $R_1$ is an entire matrix function. If we write the asymptotic condition \eqref{eq:2.2} as
\[    Y_{n,m}(z) = \left( I + \frac{A(n,m)}{z} + \mathcal{O}(1/z^2) \right)
		\begin{pmatrix} z^{n+m} & 0 & 0 \\ 0 & z^{-n} & 0 \\ 0 & 0 & z^{-m} \end{pmatrix}, \]
where $A(n,m)$ is the $3\times 3$ matrix coefficient of $1/z$ in the $\mathcal{O}(1/z)$ term of \eqref{eq:2.2}, then after some calculus
and in view of Liouville's theorem we find
\begin{equation} \label{eq:2.8}
    L_{n,m}  = \begin{pmatrix}
                  z+A_{1,1}(n+1,m)-A_{1,1}(n,m) & -A_{1,2}(n,m) & -A_{1,3}(n,m) \\
                  A_{2,1}(n+1,m) & 0 &  0 \\
                  A_{3,1}(n+1,m) & 0 &  1
                 \end{pmatrix},
\end{equation}
where $A_{i,j}(n,m)$ is the entry on row $i$ and column $j$ of $A(n,m)$.
We can therefore write
\begin{equation} \label{eq:2.9}
      Y_{n+1,m} =L_{n,m}  Y_{n,m}.
\end{equation}
In a similar way we also have
\begin{equation} \label{eq:2.10}
      Y_{n,m+1} = M_{n,m}  Y_{n,m},
\end{equation}
with
\begin{equation} \label{eq:2.11}
    M_{n,m}  = \begin{pmatrix}
                  z+A_{1,1}(n,m+1)-A_{1,1}(n,m) & -A_{1,2}(n,m) & -A_{1,3}(n,m) \\
                  A_{2,1}(n,m+1) & 1 &  0 \\
                  A_{3,1}(n,m+1) & 0 &  0
                 \end{pmatrix}.
\end{equation}
Now, introducing
\begin{equation} \label{cd}
    c_{n,m} = A_{1,1}(n,m)-A_{1,1}(n+1,m), \quad
    d_{n,m} = A_{1,1}(n,m)-A_{1,1}(n,m+1)
\end{equation}
and
\begin{equation} \label{ab}   a_{n,m} = c_1(n,m)A_{1,2}(n,m), \quad b_{n,m} = c_2(n,m)A_{1,3}(n,m)\end{equation}
we see that the $(1,1)$-entry of \eqref{eq:2.9} gives the first relation in \eqref{0 RR}
$$ P_{n+1,m}(x) = (x-c_{n,m})P_{n,m}(x) - a_{n,m} P_{n-1,m}(x) - b_{n,m} P_{n,m-1}(x),$$
and \eqref{eq:2.10} gives the second relation in \eqref{0 RR}
$$P_{n,m+1}(x) = (x-d_{n,m})P_{n,m}(x) - a_{n,m} P_{n-1,m}(x) - b_{n,m} P_{n,m-1}(x). $$
\end{proof}

\subsection{Discrete integrable systems represented by $3\times 3$ matrices}\label{sec:42}
Another immediate consequence of the reformulation of \HP{} approximation in terms of Riemann-Hilbert problems is a bridge between the corresponding recurrence relations and discrete integrable system whose transition matrices are $3\times 3$ matrices.
\begin{proposition}\label{Pr2}
Let $(f_1,f_2)$ be a perfect system., i.e., all the multi-indices $(n,m)$ are normal. 
Then there exists a wave function \eqref{eq:2.3} on $\dZ_+^2$ and its
transition matrices are given by \eqref{Int_eq:2.8}:
\begin{equation*}
    L_{n,m}  = \begin{pmatrix}
                  x+ \alpha_{n,m}^{(1)}& \alpha_{n,m}^{(2)} & \alpha_{n,m}^{(3)} \\
                  \alpha_{n+1,m}^{(4)} & 0 &  0 \\
                  \alpha_{n+1,m}^{(5)} & 0 &  1
                 \end{pmatrix},\qquad
    M_{n,m}  = \begin{pmatrix}
                  x+\beta_{n,m}^{(1)} & \alpha_{n,m}^{(2)} & \alpha_{n,m}^{(3)} \\
                  \alpha_{n,m+1}^{(4)} & 1 &  0 \\
                  \alpha_{n,m+1}^{(5)} & 0 &  0
                 \end{pmatrix},
\end{equation*}
whose entries are related to the coefficients of the recurrence relations \eqref{0 RR} for the \HP{} polynomials of the
functions $f_1$ and $f_2$ as follows:
\begin{equation} \label{abcd}
    c_{n,m} = -\alpha_{n,m}^{(1)}, \quad
    d_{n,m} = -\beta_{n,m}^{(1)}, \quad
\   a_{n,m} = -\alpha_{n,m}^{(4)}\alpha_{n,m}^{(2)}, \quad b_{n,m} = -\alpha_{n,m}^{(5)}\alpha_{n,m}^{(3)}.
\end{equation}
\end{proposition}

\begin{proof}
We take a function $Y$ of the form \eqref{eq:2.3}, then \eqref{eq:2.9} and \eqref{eq:2.10} give us
transition matrices $L_{n,m},\,M_{n,m}$ of the form  \eqref{eq:2.8} and \eqref{eq:2.11}.
Taking into account that the normalizing factors in \eqref{ab} are
$$ c_1(n,m) = A_{2,1}(n,m)\quad \mbox{and} \quad c_2(n,m) = A_{3,1}(n,m),$$
the relations \eqref{ab} and \eqref{cd} give \eqref{abcd}.
Finally, we notice that to prove \eqref{eq:2.9} and \eqref{eq:2.10} we only used the fact that $Y$
admits the following factorization
\[
Y(z)=R(z)\begin{pmatrix}
                        1 &\oint_{\Gamma} \frac{f_1(x)}{x-z}\,dx& \oint_{\Gamma} \frac{f_2(x)}{x-z}\,dx\\
                        0 &  1  & 0  \\
                        0 & 0 & 1
                        \end{pmatrix},
\]
where $R$ is a matrix polynomial.

\end{proof}

\begin{remark}
As we saw in Section 3.1, a continued fraction is just a sequence of $2\times 2$ matrices whose determinants  equal 1. 
Now we see that starting with a perfect system of two functions one can get a pair of two-dimensional sequences of $3\times 3$ matrices, which is a certain two-dimensional generalization of continued fractions. So, the scheme to find transition matrices is actually a certain two-dimensional generalization of the Schur-Euclid algorithm:
\begin{enumerate}
    \item[(i)]
        We start with the function
        \[
       Y(z)=Y_0(z)=\begin{pmatrix}
                        1 &\oint_{\Gamma} \frac{f_1(x)}{x-z}\,dx& \oint_{\Gamma} \frac{f_2(x)}{x-z}\,dx\\
                        0 &  1  & 0  \\
                        0 & 0 & 1
                        \end{pmatrix},
         \]
         where $f_1$ and $f_2$ are Laurent series \eqref{0 f};
    \item[(ii)]
        Having constructed $Y_{n,m}$, we look for the transition matrices $L_{n,m}$ and $M_{n,m}$ of the
        form \eqref{eq:2.8} and \eqref{eq:2.11}, such that
        the functions
        \[
        Y_{n+1,m} =L_{n,m}  Y_{n,m},\quad
        Y_{n,m+1} = M_{n,m}  Y_{n,m},
        \]
        obey the corresponding asymptotic condition \eqref{eq:2.2}.
\end{enumerate}
Note that the transition matrices in step (ii) are uniquely determined due to the construction. In Section 5 this idea will be further developed and conventional continued fractions will appear there.
\end{remark}

Now we simplify the zero curvature condition
\[
0 = L_{n,m+1} M_{n,m} -  M_{n+1,m} L_{n,m},
\]
to the form of \eqref{0 DiffEq}.
\begin{proof}[Proof of Theorem~\ref{T0 2}] In \cite{vanA2011} the consistency condition for the  recurrence coefficients of \eqref{0 RR} was obtained in the following form:
\smallskip

\begin{equation} \label{eq:3.1}\begin{cases}
    d_{n+1,m}-d_{n,m} = c_{n,m+1} - c_{n,m}, \\
    b_{n+1,m}-b_{n,m+1} + a_{n+1,m}-a_{n,m+1} = \det \begin{pmatrix} d_{n+1,m} & d_{n,m} \\
	                                                             c_{n,m+1} & c_{n,m} \end{pmatrix},
  \\
   \displaystyle \frac{a_{n,m+1}}{a_{n,m}} = \frac{c_{n,m}-d_{n,m}}{c_{n-1,m}-d_{n-1,m}},  \\
   \displaystyle \frac{b_{n+1,m}}{b_{n,m}} = \frac{c_{n,m}-d_{n,m}}{c_{n,m-1}-d_{n,m-1}}.
\end{cases}\end{equation}
\smallskip

Using the first equation in \eqref{eq:3.1}, we subtract the columns of the determinant of the second equation in \eqref{eq:3.1}. We thus obtain the first two equations of \eqref{0 DiffEq}. The third and fourth equations of \eqref{0 DiffEq} and \eqref{eq:3.1} are the same.
\end{proof}

\begin{remark}
There are other systems related to the concept of orthogonality for which the consistency leads to non-trivial zero curvature conditions  \cite{PGR1995}, \cite{SNderK2011}, \cite{SpZh}.
\end{remark}

It turns out that the consistency conditions \eqref{eq:3.1} (or, equivalently, the zero curvature condition) are also sufficient for a sequence of \HP{} polynomials
to exist and correspond to a perfect system of functions. To complete the proof of  Theorem~\ref{T0 2} it remains to prove the following result.

\begin{proposition}\label{Prop4.5}
Suppose that the zero curvature condition
\[
0 = L_{n,m+1} M_{n,m} -  M_{n+1,m} L_{n,m},
\]
holds for a family of invertible matrices $L_{n,m}$ and $M_{n,m}$ of the form \eqref{eq:2.8} and \eqref{eq:2.11}.
Then there are two functions $f_1$ and $f_2$ such that the polynomials $P_{n,m}$ satisfying the corresponding
relations \eqref{0 RR} are the \HP{} polynomials for $f_1$ and $f_2$.
\end{proposition}

\begin{proof}
To determine the functions we first construct the polynomials $P_{n,0}$ and $P_{0,m}$. This can be done since they satisfy ordinary three-term recurrence relations.
So these polynomials are orthogonal polynomials due to the spectral theorem for orthogonal polynomials \cite[\S 2.5]{Ismail}. Let $f_1$ be the function corresponding to $P_{n,0}$
and let $f_2$ be the function for $P_{0,m}$. Next, the  consistency conditions \eqref{eq:3.1}
allow us to define $Y_{n,m}$ in a unique way for all pairs $(n,m)\in\dZ_+^2$. Due to the asymptotic condition \eqref{eq:2.2},
the first column of $Y_{n,m}$ consists of \HP{} polynomials. At the same time, these polynomials
coincide with $P_{n,m}$. Some more details on how to reconstruct the sequence $P_{n,m}$ from the marginal orthogonal
polynomials are given in \cite{FHVanA}.
\end{proof}

To conclude this subsection, note that the wave function $\Psi_{n,m}$ 
coincides with $Y_{n,m}$ for $(n,m)\in\dZ_+^2$ and it can be extended to 
the entire lattice $\dZ^2$ by the symmetry
\begin{equation}\label{SymW}
\Psi_{-n,m}=\Psi_{n,m}, \quad \Psi_{n,-m}=\Psi_{n,m},\quad \Psi_{-n,-m}=\Psi_{n,m}, \qquad n,m\in\dZ_+.
\end{equation}

\section{The underlying boundary value problem} 

In this section we discuss the discrete system and give a few algorithms of reconstructing solutions from boundary data. Finally we will consider some observable classes of initial data for which the system is solvable. 

Here we are concerned  with the question of finding a solution of the difference equations
\begin{equation}  \label{0 DiffEqD}
\begin{cases}
    \,\,\,c_{n,m+1}\,=\, c_{n,m}\,+\,\displaystyle\frac{(a+b)_{n+1,m}\,-\,(a+b)_{n,m+1}}{(c-d)_{n,m}},  \\
    \,\,\,d_{n+1,m}\,=\, d_{n,m}\,+\,\displaystyle\frac{(a+b)_{n+1,m}\,-\,(a+b)_{n,m+1}}{(c-d)_{n,m}},  \\
    \,\,\,a_{n,m+1}\,=\,a_{n,m} \,\displaystyle\frac{(c-d)_{n,m}}{(c-d)_{n-1,m}}, \\
    \,\,\,b_{n+1,m}\,=\,b_{n,m} \,\displaystyle\frac{(c-d)_{n,m}}{(c-d)_{n,m-1}},
\end{cases} \qquad n,m \geq 0\,,
\end{equation}
subject to the boundary conditions
\begin{equation}\label{BVP}
\begin{split}
 a_{0,m} = 0, \quad d_{0,m}=\hat{d}_m,\quad b_{0,m+1}=\hat{b}_{m+1}, \quad m\in\dZ_+, \\ 
 b_{n,0}=0, \quad c_{n,0}=\hat{c}_n, \quad
   a_{n+1,0}=\hat{a}_{n+1}, \quad n\in\dZ_+,
\end{split}
\end{equation}
where $\hat{c}_n$,  $\hat{d}_m$, $\hat{a}_{n+1}$, and $\hat{b}_{m+1}$ are sequences of complex numbers. More precisely, we are interested in solutions for which one has
\[
a_{n,m}\neq 0, \quad b_{n,m}\neq 0, \quad n,m > 0.
\]
According to Theorem \ref{T0 2}, such a solution exists if and only if there is a perfect system of two functions $f_1$ and $f_2$. In this case the initial data $\hat{c}_n$, $\hat{a}_{n+1}$, $\hat{d}_m$, and $\hat{b}_{m+1}$ are the entries of the $J$-fraction representations of $f_1$ and $f_2$:
\[
f_1(z)\sim
-\frac{1}{\displaystyle{z-\hat{c}_0-\frac{\hat{a}_1}{\displaystyle{z-\hat{c}_1-\frac{\hat{a}_2}{\ddots}}}}},\quad
f_2(z)\sim
-\frac{1}{\displaystyle{z-\hat{d}_0-\frac{\hat{b}_1}{\displaystyle{z-\hat{d}_1-\frac{\hat{b}_2}{\ddots}}}}},
\]
where we have
\begin{equation}\label{nondeg}
a_{n+1,0}=\hat{a}_{n+1}\ne 0, \quad b_{0,m+1}=\hat{b}_{m+1}\neq 0.
\end{equation}
Basically, this means that in order to have a well-posed boundary value problem, the initial data have to be generated through the Schur-Euclid algorithm by two functions that form a perfect system, i.e., a system of two functions that determines the entire table of multiple orthogonal polynomials. 

One of the main obstacles to construct a table of multiple orthogonal polynomials is to ensure that each index is normal, that is, the corresponding determinant $S_{n,m}$ is non-zero. {\bf It cannot be done for two arbitrary analytic functions or even for any two positive measures} and this issue was addressed for the first time by K. Mahler \cite{Mah}, who coined the notion of perfect systems. 
To be more precise, a system of two measures is called perfect if each index in the corresponding table is normal, i.e., $S_{n,m}\ne 0$ for all $n,m\in\dZ_+$. At the end of this section we give two rather general classes of perfect systems. In turn, these systems give rise to an infinite number of solutions of the discrete integrable system in question. Before going into more details about those classes, we reformulate a part of Theorem \ref{T0 2} and give a constructive proof, which actually is the way to solve \eqref{0 DiffEqD}, \eqref{BVP}.
\begin{proposition}
If the initial data $\hat{c}_n$,  $\hat{d}_m$, $\hat{a}_{n+1}$, and $\hat{b}_{m+1}$ correspond to a perfect system, then the boundary value problem
\eqref{0 DiffEqD}, \eqref{BVP} has a solution that satisfies the condition \eqref{nondeg}. Moreover, the problem can be solved in the following way. The boundary data $\hat{c}_n$,  $\hat{d}_m$, $\hat{a}_{n+1}$, and $\hat{b}_{m+1}$ define the moments and the solution $(c,\,a)_{n,m}$, $(d,\,b)_{n,m}$ can be recovered from the moments.
\end{proposition}
\begin{proof}
The statement is a straightforward consequence of Theorem \ref{T0 2}. So, we know that the underlying system of functions $f_1$ and $f_2$ is perfect, that is, we have that the corresponding moments are such that $S_{n,m}\ne 0$ for all $n,m\in\dZ_+$. As a matter of fact, it is a standard technique that recovers the moments from the entries of the corresponding $J$-fraction (see \cite{Ach1961}). Furthermore, using the moments one can reconstruct the solution in the following manner. The coefficients $a_{n,m}$ and $b_{n,m}$ can be found via the formulas from \cite{vanA2011}
\begin{equation*}
a_{n,m} = \frac{S_{n+1,m} S_{n-1,m}}{\big(S_{n,m}\big)^2}, \quad  
b_{n,m} = \frac{S_{n,m+1} S_{n,m-1}}{\big(S_{n,m}\big)^2}. 
\end{equation*} 
We also know from \cite{vanA2011} that 
\begin{equation}\label{DminusC}
d_{n,m}-c_{n,m} = \frac{S_{n,m} S_{n+1,m+1}}{S_{n+1,m} S_{n,m+1}}.
\end{equation}  
Finally, the rest are found by summation of the first and second equations of the system for consecutive indices
\begin{equation*}
\begin{split}
c_{n,m+1}\,=\, c_{n,0}\,+\sum_{i=1}^m\displaystyle\frac{(a+b)_{n+1,i}\,-\,(a+b)_{n,i+1}}{(c-d)_{n,i}},  \\
d_{n+1,m}\,=\, d_{0,m}\,+\sum_{i=1}^n\displaystyle\frac{(a+b)_{i+1,m}\,-\,(a+b)_{i,m+1}}{(c-d)_{i,m}}, 
\end{split}
\end{equation*}
since the elements on the right hand sides are already determined or are part of the initial data.
\end{proof}

Sometimes, given a perfect system, it is easier to find the \HP{} polynomials rather then moments and for this reason we can elaborate a bit more on the idea mentioned in Remark 4.3 in order to see some continued fractions in the conventional sense. Following \cite{GS} we introduce the following functions
\[
m_-^{(1)}(z,n,m)=-\frac{P_{n,m}(z)}{P_{n+1,m}(z)}, \quad
m_-^{(2)}(z,n,m)=-\frac{P_{n,m}(z)}{P_{n,m+1}(z)},
\]
which can serve as a tool for solving the system \eqref{0 DiffEq}.
\begin{proposition}\label{branchS} 
If the initial data $\hat{c}_n$,  $\hat{d}_m$, $\hat{a}_{n+1}$, and $\hat{b}_{m+1}$ correspond to a perfect system, then the boundary value problem
\eqref{0 DiffEqD}, \eqref{BVP} has a solution that satisfies the condition \eqref{nondeg}.
Furthermore, the boundary value problem  \eqref{0 DiffEqD}, \eqref{BVP} can be solved in the following way. 
From the initial data $(c,\,a)_{n,0}$, $(d,\,b)_{0,m}$ one can reconstruct the family of polynomials $P_{n,m}$
(which can formally be done  via formula \eqref{eq:2.12_1}). In turn, these polynomials define the functions $m_-^{(1)}(z,n,m)$ and $m_-^{(2)}(z,n,m)$, which admit
the following continued fraction expansions 
\[ 
m_-^{(1)}(z,n,m)=
-\frac{1}{\displaystyle z-c_{n,m}-\frac{a_{n,m}}{z-c_{n-1,m}-\dots}-\frac{b_{n,m}}{z-d_{n,m-1}-\dots}},
\]
\[
m_-^{(2)}(z,n,m)=-\frac{1}{\displaystyle z-d_{n,m}-\frac{a_{n,m}}{z-c_{n-1,m}-\dots}-\frac{b_{n,m}}{z-d_{n,m-1}-\dots}}
\]
that determine the solution $(c,\,a)_{n,m}$, $(d,\,b)_{n,m}$ to the equation \eqref{0 DiffEq}. 
\end{proposition}

\begin{proof}
It is not so hard to see that the relations \eqref{0 RR} can be rewritten as the following generalization of the discrete Riccati equation:
\begin{align}
m_-^{(1)}(z,n,m)=\frac{-1}{z-c_{n,m}+a_{n,m}m_-^{(1)}(z,n-1,m)+b_{n,m}m_-^{(2)}(z,n,m-1)}, \label{Mminus1}\\
m_-^{(2)}(z,n,m)=\frac{-1}{z-d_{n,m}+a_{n,m}m_-^{(1)}(z,n-1,m)+b_{n,m}m_-^{(2)}(z,n,m-1)}. \label{Mminus2}
\end{align}
Then, we see that applying \eqref{Mminus1} and \eqref{Mminus2} iteratively leads to the continued fraction expansions under consideration, which allow us to recursively reconstruct the solution from the initial boundary data. Namely, it is clear how to reconstruct $c_{n,m}$ and $d_{n,m}$ for all indices from $m_-^{(1)}(z,n,m)$ and $m_-^{(2)}(z,n,m)$ in the first place. Then the first term of the asymptotic expression
\[
-\frac{1}{m_-^{(1)}(z,n,m)}-z+c_{n,m}, \quad z\to \infty
\]
(or, equivalently, the analogous one for $m_-^{(2)}(z,n,m))$ determines $a_{n,m}+b_{n,m}=f_{n,m}$ and the consecutive term gives
$a_{n,m}c_{n-1,m}+b_{n,m}d_{n,m-1}=g_{n,m}$. Next, taking into account the first equation in \eqref{eq:3.1} we have that
\[
c_{n-1,m}-d_{n,m-1}=c_{n-1,m-1}-d_{n-1,m-1}.
\] 
Since the system is perfect we get that $c_{n-1,m-1}-d_{n-1,m-1} \neq 0$ (see \eqref{DminusC}). Therefore, $a_{n,m}$ and $b_{n,m}$ are uniquely determined.
\end{proof}

Let us emphasize here that the above given continued fractions branch into two continued fractions on each level. Actually, what we see is that there are two fractions behind the scene. On the one hand, they are similar to the classical one in \eqref{toGen} but, on the other hand, they have a certain two-dimensional structure and the two fractions are identical except for one entry. 

\begin{remark} There is one more algorithm available, which is obtained by combining the well-known  Jacobi-Perron algorithm (see \cite{NS} for the details) and a result from \cite{FHVanA}. Namely,  the Jacobi-Perron algorithm expands the vector $(f_1,f_2)$, where $f_1$ and $f_2$ form a perfect system, into a vector continued fraction. The approximants of this vector continued fraction consists of rational functions with the same denominator. More importantly, the denominators  are the \HP{} polynomials that correspond to the so-called step-line, that is when $m=n$ and $m=n-1$. Furthermore, the step-line polynomials satisfy recurrence relations whose coefficients are the input for the algorithm given in \cite{FHVanA} to reconstruct all the coefficients of the nearest neighbour recurrence relations from the step-line. 
\end{remark}

\subsection{Angelesco systems.} A. Angelesco considered in \cite{Ang} the following systems of measures. Let $\Delta_1$ and $\Delta_2$ be disjoint bounded intervals on the real line and $(\mu_1,\mu_2)$ be a system of measures such that
$\supp \mu_j = \Delta_j$.

Fix $\vec{n} = (n_1,n_2) \in {\dZ}_+^{2}$ and consider the multiple orthogonal polynomials of the so called Angelesco system $(\widehat{\mu}_1,\widehat{\mu}_2)$ relative to $\vec{n}$. Here  $\widehat{\mu}$ denotes the Cauchy transform of $\mu$:
\[
\widehat{\mu}(z) = \int \frac{d\mu(x)}{z-x}.
\]
By construction, we have that
\[ \int x^{\nu} P_{\vec{n}}(x) \, d\mu_j(x) =0, \qquad \nu = 0,\ldots,n_j -1,\quad j=1,2.
\]
Therefore, $P_{\vec{n}}$ has $n_j$ simple zeros in the interior (with respect to the Euclidean topology of ${\dR}$) of $\Delta_j$. As a consequence, since the intervals $\Delta_j$ are disjoint, $\deg P_{\vec{n}} = |\vec{n}| = n_1+n_2$ and
\textit{Angelesco systems are perfect}.

\subsection{Nikishin systems}
Unfortunately, Angelesco's paper received little attention and such systems reappeared many years later in \cite{Nik1} where E.M. Nikishin  deduced some of their formal properties.
He also introduced another class of systems for which the perfectness was proved only recently \cite{Nik}.

To get an idea about these systems, let us consider two disjoint bounded intervals $\Delta_{1}, \Delta_{2}$ on the real line. Suppose we are given two measures $\sigma_{1}$ and $\sigma_{2}$ supported
on $\Delta_{1}$ and $\Delta_{2}$, respectively. With these two measures we define a third one in the following way
\[ d\left<\sigma_{1},\sigma_{2}\right>(x) = \widehat{\sigma}_{2}(x)\, d\sigma_{1}(x);
\]
that is, one multiplies the first measure by a weight formed by the Cauchy transform of the second measure. Thus, we have arrived at the notion of a Nikishin system. A system of two measures $(\mu_1, \mu_2)$ of the form
\[
d\mu_1(x)=d\sigma_{1}(x), \qquad d\mu_2(x)=\int \frac{d\sigma_1(t)}{x-t}\, d\mu_1(x)
=d\left<\sigma_{1},\sigma_{2}\right>
\]
is called a Nikishin system (of order 2). In fact, a Nikishin system can be defined for any finite number of measures.
We also emphasize that the measures from a Nikishin system have the same support,
which is a totally different situation than in the case of an Angelesco system.

Finally, it is worth mentioning here that it was a long standing problem to prove that \textit{a general Nikishin system
$(\mu_1, \mu_2, \dots, \mu_p)$ is perfect for $p\ge 2$}. This fact was finally proved
in the remarkable paper \cite{FL}. 




\end{document}